\documentclass[ 12pt]{article}
\usepackage[top=1in, bottom=1in, left=1in, right=1in]{geometry}
\usepackage{amssymb}
\usepackage{amsthm}
\usepackage{amsfonts}
\usepackage{amsmath}
\usepackage{hyperref}
\usepackage{bbm}
\usepackage{color}

\newtheorem*{theorem*}{Theorem}

\newtheorem{theorem}{Theorem}[section]

\newtheorem{lemma}[theorem]{Lemma}
\newtheorem{corollary}[theorem]{Corollary}

\newtheorem{remark}[theorem]{Remark}

\newtheorem{proposition}[theorem]{Proposition}

\def\del{\partial}
\def\dbar{\bar\partial}
\def\ddbar{\del\dbar}

\def\del{\partial}

\def\o{\omega}

\def\beq{\begin{equation}}
\def\eeq{\end{equation}}

\title{The isometries of the space of K\"ahler metrics}

\author{Tam\'as Darvas}
\date{\emph{\small{To Anita. \vspace{-0.1in}}}}
\begin{document}
\maketitle
\begin{abstract} Given a compact K\"ahler manifold, we prove that all global isometries of the space of K\"ahler metrics are induced by biholomorphisms and anti-biholomorphisms of the manifold. In particular, there exist no global symmetries for Mabuchi's metric. Moreover, we show that the  Mabuchi completion does not even admit local symmetries. Closely related to these findings, we provide a large class of metric geodesic segments that can not be extended at one end, pointing out the first such examples in the literature.
\end{abstract}

\section{The main results}

Let $(X,\omega)$ be a compact connected K\"ahler manifold. Given a K\"ahler metric $\omega'$ cohomologuos to $\omega$, by the $\partial \bar \partial$-lemma of Hodge theory there exists $u \in C^\infty(X)$ such that $\omega' := \omega + i\ddbar u.$
Such a metric $\omega'$ is said to belong to the \emph{space of K\"ahler metrics} $\mathcal H$. By the above, up to a constant, one can identify $\mathcal H$ with the \emph{space of K\"ahler potentials}:
$$\mathcal H_\o := \{u \in C^\infty(X) \ \textup{ s.t. } \ \omega + i\ddbar u >0 \}.$$ 
This space can be endowed with  a natural infinite dimensional $L^2$ type Riemannian metric \cite{Ma87,Se92,Do99}:
\begin{equation}\label{eq: Riem}
\langle\xi,\zeta\rangle_{v} := \frac{1}{V}\int_X \xi \zeta \omega_v^n, \ \ v \in \mathcal H_\omega, \ \xi,\zeta \in T_v \mathcal H_\o\simeq C^\infty(X),
\end{equation}
where $V = \int_X \omega^n$. Additionally, Donaldson  and Semmes pointed out that $(\mathcal H_\omega, \langle \cdot,\cdot \rangle)$ can be thought of as a formal symmetric space \cite{Se96, Do99}:
\begin{equation}\label{eq: H_simet}
\mathcal H_\omega \simeq \frac{\textup{Ham}_\omega^{\Bbb C}}{\textup{Ham}_\omega},
\end{equation}
where $\textup{Ham}_\omega$ is the group of Hamiltonian symplectomorphisms of $\omega$, and $\textup{Ham}_\omega^{\Bbb C}$ is its formal complexification. Though not quite precise, the underlying heuristic of \eqref{eq: H_simet} led to many advances in the understanding of the geometry of $\mathcal H_\o$, as well as the formulation of stability conditions aiming to characterize existence of canonical metrics (for an exposition see \cite{Sz14}).

\paragraph{Global $L^2$ isometries and  symmetries of $\mathcal H_\o$.} For finite dimensional Riemannian manifolds, the existence of a symmetric structure arising as a quotient of Lie groups, as in \eqref{eq: H_simet}, is equivalent with existence of global symmetries at all points of the manifold \cite{He01}. Such maps  are global involutive isometries reversing geodesics at a specific point. If such symmetries existed for $(\mathcal H_\o, \langle\cdot,\cdot\rangle)$ it would perhaps allow to make a precise sense of \eqref{eq: H_simet}.

Recently a large class of local symmetries of $\mathcal H_\o$ were constructed in \cite{BCKR16}, via complex Legendre transforms, that also found applications to interpolation of norms \cite{BCKR16b}. Moreover, it was shown in \cite{Le17b} that all local symmetries of $\mathcal H_\o$ arise from the construction of \cite{BCKR16}. Below we show that global symmetries actually do not exist, in particular these local symmetries can not be extended to $\mathcal H_\o$. This will follow from our characterization of the isometry group of $(\mathcal H_\o, \langle\cdot,\cdot\rangle)$.

First we recall some terminology. Let $\mathcal U,\mathcal V \subset \mathcal H_\o$ be open sets. We say that a map $F: \mathcal U \to \mathcal V$ is $C^1$, or (with slight abuse of terminology) differentiable, if $(F,F_*): \mathcal U \times C^\infty(X) \to \mathcal V \times C^\infty(X)$ is continuous as a map of  Fr\'echet spaces. Here $F_*$ is the differential of $F$  (see \cite[p. 3]{Le17c} and references therein for more details).  Moreover, $F:\mathcal U \to \mathcal U$ is a \emph{differentiable} $L^2$ \emph{symmetry at} $\phi \in \mathcal U$ if  $F^2=Id$, $F(\phi)=\phi$, $F_* |_\phi=-Id$ and
\begin{equation}\label{eq: d_p-isometry_def}
\int_X |\xi|^2 \omega_v^n = \int_X |F_*\xi|^2 \omega_{G(v)}^n, \ v \in \mathcal H_\o, \ \xi \in T_v \mathcal H_\o.
\end{equation}
If $F:\mathcal U \to \mathcal V$ is $C^1$, satisfies \eqref{eq: d_p-isometry_def} and it is bijective, then it is called a \emph{differentiable} $L^2$ \emph{isometry}. Due to infinite dimensionality, it is not yet known if differentiable $L^2$ isometries are automatically smooth \cite{Le17a}, hence the isometries we consider in this work are possibly more general than the ones in \cite{BCKR16,Le17b}.

A small class of global $L^2$ isometries has been previously known in the literature \cite[p. 16]{Le17a}. One of them is the so called   \emph{Monge--Amp\`ere flip} $\mathcal I:\mathcal H_\omega \to \mathcal H_\omega$, and is defined by the formula $\mathcal I(u) = u - 2I(u)$, where  $I:\mathcal H_\omega \to \Bbb R$ is the Monge--Amp\`ere energy: 
$$I(u) = \frac{1}{V(n+1)}\sum_{j=0}^n \int_X u \omega^j \wedge \omega_u^{n-j}.$$
The map $\mathcal I$ is involutive and its name is inspired by the fact that it flips the sign of $I$. Indeed, $I(\mathcal I(u)) = -I(u)$. 

We say that a biholomorphism $f:X \to X$ preserves  the K\"ahler class $[\omega]$ if $[f^*\omega] = [\omega]$. Similarly, an anti-biholomorphism $g:X \to X$ flips the the K\"ahler class $[\omega]$ if $[g^* \omega] = -[\omega]$. Such maps also induce a class of  global $L^2$ isometries, and we refer to  Section 2.3 for the detailed construction.

In our first main result we point out that these maps and their compositions are the only global differentiable $L^2$ isometries:

\begin{theorem}\label{thm: no_smooth_symmetry_intr} Let $F:\mathcal H_\o \to \mathcal H_\o$ be a differentiable $L^2$ isometry. Then exactly one of the following holds:\\
\noindent (i) $F$ is induced by a biholomorphism or anti-biholomorphism $f:X \to X$ that preserves or flips $[\omega]$, respectively.\\
\noindent (ii)  $F \circ \mathcal I$ is induced by a biholomorphism or anti-biholomorphism $f:X \to X$ that preserves or flips  $[\omega]$, respectively.
\end{theorem}

The space of potentials $\mathcal H_\o$ admits a Riemannian splitting $\mathcal H_\omega = \mathcal H \oplus \Bbb R$, via the Monge--Amp\`ere energy $I$. As the fixed point set of $\mathcal I$ is exactly $\mathcal H = I^{-1}(0)$, we obtain the following corollary regarding isometries of $\mathcal H$:
\begin{corollary}\label{cor: no_smooth_symmetry_intr} Let $F:\mathcal H \to \mathcal H$ be a differentiable $L^2$ isometry. Then $F$ is induced by a biholomorphism or anti-biholomorphism $f:X \to X$ that preserves or flips $[\omega]$, respectively.
\end{corollary}

The above results answer explicitly questions raised by Lempert regarding the extension property of local isometries \cite[p. 3]{Le17a}, though questions surrouding the isometry group of $(\mathcal H_\omega, \langle\cdot,\cdot\rangle)$ go back to  early work of Semmes \cite{Se92,Se96}.

Lastly, via the classification theorem of Lempert (recalled in Theorem \ref{thm: Lempert_isom}), we will see that neither of the maps in the statement of Theorem \ref{thm: no_smooth_symmetry_intr} are symmetries, immediately giving the following non-existence result for differentiable $L^2$ symmetries: 
\begin{corollary}\label{cor: no_smooth_symmetry_intr1} There exists no differentiable $L^2$ symmetry $F:\mathcal H_\o \to \mathcal H_\o$ at any $\phi \in \mathcal H_\o$.
\end{corollary}

\paragraph{Non-existence of local $L^2$ symmetries on the completions.} It was shown in \cite{Ch00} that \eqref{eq: Riem} induces a path length metric space $(\mathcal H_\o,d_2)$. By $(\mathcal E^2_\o,d_2)$ we denote the $d_2$-metric completion of this space, that can identified with a class of finite energy potentials  \cite{Da17b}.

Using density, any differentiable $L^2$ isometry $F:\mathcal H_\o \to \mathcal H_\o$ extends to a unique metric $d_2$-isometry $F: \mathcal E^2_\omega \to \mathcal E^2_\omega$. The proof of Theorem \ref{thm: no_smooth_symmetry_intr} consists of showing that contradictions arise in this extension process, unless $F$ is very special. With this and the above results in mind, one may hope  that the isometry group of the metric space $(\mathcal E^2_\o,d_2)$ could possibly admit elements beyond the ones that arise from the global differentiable $L^2$ isometries of $\mathcal H_\o$. Though this may be true,  we point out below that even local symmetries fail to exist in the context of the completion, further elaborating on phenomenon related to Corollary \ref{cor: no_smooth_symmetry_intr1}. 

Before stating our result, we recall some facts about the $d_2$-geodesics of $\mathcal E^2$. For more details we refer to Section 2.2 and the recent survey \cite{Da18}.
Let $\mathcal V \subset \mathcal E^2_\o$ be $d_2$-open with $\phi \in \mathcal V \cap \mathcal H_\o$. Given a $d_2$-geodesic $[0,1] \ni t \to \phi_t \in \mathcal V$ with $\phi_0 =\phi$, since $t \to \phi_t(x)$ is $t$-convex for almost every $x \in X$, it is possible to introduce $\dot \phi_0 = \frac{d}{dt}|_{t =0} \phi_t$. Moreover, due to \cite[Theorem 2]{Da15}, it follows that
$\dot \phi_0 \in L^2(\omega_\phi^n).$

Let $G: \mathcal V \to G(\mathcal V) \subset \mathcal E^2_\o$ be an $L^2$ isometry, i.e, a bijective map satisfying  $d_2(v_1,v_2)=d_2(G(v_1),G(v_2)), \ v_1,v_2 \in \mathcal V$. It is clear that in this case $t \to G(\phi_t)$ is also a $d_2$-geodesic. Furthermore, we say that $G$ is a \emph{metric} $L^2$ \emph{symmetry at} $\phi$ if $G^2= Id$, $G(\phi)=\phi$ and $\dot {G(\phi_0)} = - \dot \phi_0,$ i.e., $G$ ``reverses"  $d_2$-geodesics at $\phi$.

 Unfortunately,  metric $L^2$ symmetries actually do not exist, implying that the analog of \cite[Theorem 1.2]{BCKR16} does not hold in the context of the metric completion, answering questions of Berndtsson and Rubinstein \cite{Ru19}:

\begin{theorem}\label{thm: no_metric_symmetry_intr} Let $\mathcal V \subset \mathcal E^2_\o$ be a $d_2$-open set and $\phi \in \mathcal V \cap \mathcal H_\o$. There exists no metric $L^2$ symmetry $F: \mathcal V \to \mathcal V$ at $\phi$. 
\end{theorem}

Given that  $(\mathcal E^2_\o,d_2)$ is CAT(0), the group of isometries of this metric space has special structure \cite{BH99}, as pointed by B. McReynolds during the Ph.D. thesis defense of the author. In light of the above result, we expect that the group of metric isometries can be characterized as in Theorem \ref{thm: no_smooth_symmetry_intr}, though this remains an open question.

\paragraph{The extension property of geodesic segments.}
As an intermediate step in the proof of Theorem \ref{thm: no_metric_symmetry_intr} we show that a large class of $d_2$-geodesic segments inside $\mathcal E^2_\o$ can not be extended at one of the endpoints. Previously no such examples were known.

\begin{theorem}\label{thm: no_geod_ext_intr} Let $\phi_0 \in \mathcal H_\o$ and $\phi_1 \in \mathcal E^2_\o \setminus L^\infty$. Then the $d_2$-geodesic $t \to \psi_t$ connecting these potentials can not be extended to a $d_2$-geodesic $(-\varepsilon,1] \ni t \to \phi_t \in \mathcal E^2_\o$ for any $\varepsilon >0$.
\end{theorem}

For finite dimensional manifolds, topological and geodesical completeness are equivalent due to the classical Hopf--Rinow theorem. According to the above result, this is not the case for the completion $(\mathcal E^2_\o, d_2)$, despite the fact that this space it is non-positively curved \cite{CC02, Da17b}.

It will be interesting to see if a similar property holds for the $C^{1,1}$-geodesics of Chen and Chu--Tosatti--Weinkove, joining the potentials of $\mathcal H_\o$ \cite{Ch00,CTW17}. 

\paragraph{Relation to the $L^p$ geometry of $\mathcal H_\omega$.} In \cite{Da15} the author introduced a family of $L^p$ Finsler metrics on $\mathcal H_\o$ for any $p \geq 1$, generalizing \eqref{eq: Riem}:
$$\|\xi\|_{p,v} = \bigg(\frac{1}{V}\int_X |\xi|^p \omega_v^n\bigg)^{\frac{1}{p}}, \ \ v \in \mathcal H_\o, \ \xi \in T_v \mathcal H_\o.$$
These induce path length metric spaces $(\mathcal H_\o,d_p)$, and in \cite{Da15} the author computed the corresponding metric completions, that later  found applications to existence of canonical metrics (for a survey see \cite{Da18}). Though this more general context lacks the symmetric space interpretation, all of our above results can be considered in the $L^p$ setting as well.

As the reader will be able to deduce from our arguments below, the $L^p$ version of Theorem \ref{thm: no_metric_symmetry_intr} holds for any $p > 1$. Our proof does not work when $p=1$, since the class of finite energy geodesics may not be stable under isometries  in this case (see \cite[Theorem 1.2]{DL18}). On the other hand, the $L^p$ version of Theorem \ref{thm: no_geod_ext_intr} does hold for all $p \geq 1$. Lastly, our argument for Theorem \ref{thm: no_smooth_symmetry_intr} would most likely go through in the $L^p$ context in case one could obtain the analog of Theorem \ref{thm: Lempert_isom} for differentiable $L^p$ isometries.

\paragraph{Acknowledgements.} We thank L. Lempert for extensive feedback on our manuscript, and for generously explaining to us details about his paper \cite{Le17a}. We also thank B. Berndtsson and Y.A. Rubinstein for suggestions on how to improve the presentation. This work was partially supported by NSF grant DMS 1610202.

\section{Preliminaries}

For simplicity we assume throughout the paper the the K\"ahler metric $\omega$ satisfies the following volume normalization:
$$V = \int_X \omega^n = 1.$$
Using a dilation of $\omega$ this can always be achieved and does not represent loss of generality.

\subsection{The classification theorem of Lempert} 

In this short section we recall the particulars of a result due to Lempert on the classification of local $C^1$ isometries on $\mathcal H_\o$ (\cite[Theorem 1.1]{Le17a}), tailored to our global setting:
\begin{theorem}\label{thm: Lempert_isom} Suppose that $F:\mathcal H_\omega \to \mathcal H_\omega$ is a differentiable $L^2$ isometry. Then for $u \in \mathcal H_\o$ there exists a unique $C^\infty$ diffeomorphism $G_u: X \to X$ such that $G_u^* \omega_u = \pm  \omega_{F(u)}$ and
\begin{equation}\label{eq: Lempert_formula}
F_*(u)\xi = a \xi \circ G_u - b \int_X \xi \omega_u^n, \ \xi \in T_u\mathcal H_\o \simeq C^\infty(X),
\end{equation}
where $a =1$, or $a = -1$, or $b=0$, or $b=2a$.
\end{theorem}
In the particular case of the (local) $L^2$ symmetries constructed in \cite{BCKR16}, formula \eqref{eq: Lempert_formula} is a consequence of \cite[Theorem 5.1, Theorem 6.1, Proposition 7.1]{BCKR16} with $a=-1$ and $b=0$. 

\begin{remark} It follows from the proof of \cite[Theorem 1.1]{Le17a} that the integers $a$ and $b$ in the statement depend continuously on $u \in \mathcal H_\omega$ (as does $G_u$), hence in our case they are independent of $u$, as $\mathcal H_\omega$ is connected. This was pointed out to us by L. Lempert \cite{Le19}.
\end{remark}

 From the classification theorem we obtain the following simple monotonicity result:
\begin{proposition}\label{prop: monotonicity} Suppose that $F:\mathcal H_\omega \to \mathcal H_\omega$ is a differentiable $L^2$ isometry with $b=0$. Let $c \in \Bbb R$ and $u,v \in \mathcal H_\o$ with $u \leq v$. Then the following hold:\\
\noindent (i) if $a = 1$ then $F(u) \leq F(v)$ and $F(u + c)= F(u) + c$.\\
\noindent (ii) if $a = -1$ then $F(u) \geq F(v)$ and $F(u + c)= F(u) - c$.
\end{proposition}
\begin{proof} We only address (ii), as the proof of (i) is analogous.  Let $[0,1] \ni t \to \gamma_t:= v + t (u-v)  \in \mathcal H_\o$. Then $t \to F(\gamma_t)$ is a $C^1$ curve connecting $F(v)$ and $F(u)$. Moreover, Theorem \ref{thm: Lempert_isom} implies that
$$F(u)-F(v) = \int_0^1 \frac{d}{dt} {F(\gamma_t)}dt = \int_0^1 -(u-v)\circ G_{\gamma_t} dt \geq 0.$$
The fact that $F(u + c)= F(u) - c$, follows after another application of Theorem \ref{thm: Lempert_isom} to the curve $[0,1] \ni  t \to \eta_t:= u + tc \in \mathcal H_\o$.
\end{proof}

\begin{corollary}\label{cor: G_constant} Suppose that $F:\mathcal H_\omega \to \mathcal H_\omega$ is a differentiable $L^2$ isometry with $b=0$. Then, in the language of Theorem \ref{thm: Lempert_isom} applied to $F$, we have that $G_{u+c}=G_u$ for all $u \in \mathcal H_\o$ and $c \in \Bbb R$. 
\end{corollary}
\begin{proof} We only address the case $a=1$, as the argument for $a=-1$ is identical. Let $\xi \in C^\infty(X)$. By Proposition \ref{prop: monotonicity}(i) and Theorem \ref{thm: Lempert_isom} we have that
$$\xi \circ G_{u+c}= F_*(u+c)\xi = \frac{d}{dt}\bigg|_{t=0} F(u+ t \xi + c) = \frac{d}{dt}\bigg|_{t=0} F(u+ t \xi) = F_*(u)\xi = \xi \circ G_u.$$
Since $\xi \in C^\infty(X)$ is arbitrary, we obtain that $G_{u+c}=G_u$.
\end{proof}

\subsection{The complete metric space $(\mathcal E^2_\o,d_2)$} 
In this short subsection we recall aspects from the work of the author related to the metric completion of $(\mathcal H_\o,d_2)$. For details we refer to the survey \cite{Da18}.

As conjectured by V. Guedj \cite{Gu14}, $\overline{(\mathcal H_\o,d_2)}$ can be identified with $(\mathcal E^2_\o,d_2)$, where $\mathcal E^2_\o \subset \textup{PSH}(X,\omega)$ is an appropriate subset of $\omega$-plurisubharmonic potentials \cite[Theorem 1]{Da17b}. Moreover, $(\mathcal E^2_\o,d_2)$ is a non-positively curved complete metric space, whose points can be joined by unique $d_2$-geodesics. 

Given $u_0,u_1 \in \mathcal E^2_\o$, the unique $d_2$-geodesic $[0,1] \ni t \to u_t \in \mathcal E^2_\o$ connecting these points has special properties. To start, we recall that this curve arises as the following envelope:
\begin{equation}\label{eq: geod_env_def}
u_t:= \sup\{v_t \ | \ \textup{where } t \to v_t \textup{ is a subgeodesic}\}, \ \ t \in (0,1).
\end{equation}
Here a subgeodesic $(0,1) \ni t \to v_t \in \textup{PSH}(X,\omega)$ is a curve satisfying $\limsup_{t \to 0,1} v_t \leq u_{0,1}$ and $u(s,x) := u_{\textup{Re }s}(x) \in \textup{PSH}(S \times X, \omega)$, where $S =\{0 < \textup{Re }s < 1\} \subset \Bbb C$.

It follows from \eqref{eq: geod_env_def} that $t \to u_t(x), t \in (0,1)$ is convex for all $x \in X$ away from a set of measure zero. On the complement we have that $u_t(x)=-\infty, \ t \in (0,1)$. Moreover, due to \cite[Corollary 7]{Da17b}, we also have that 
\begin{equation}\label{eq: limit_t_0_1}
\lim_{t \to 0} u_t(x) = u_{0}(x) \ \ \textup{ and } \ \ \lim_{t \to 1} u_t(x) = u_{1}(x)
\end{equation}
 for all $x \in X$ away from a set of measure zero. In the particular case when $u_0,u_1 \in \mathcal H_\o$,  the curve $t \to u_t$ is $C^{1,1}$ on $[0,1] \times X$ \cite{Ch00, Bl12, CTW17}. 

By $\mathcal C_\o$ we denote the set of continuous potentials in $\textup{PSH}(X,\omega)$. As pointed out previously, a differentiable $L^2$ isometry $F:\mathcal H_\o \to \mathcal H_\o$ induces a unique $d_2$-isometry $F:\mathcal E^2_\o \to \mathcal E^2_\o$, extending the original map (using density). Going forward, we do not distinguish $F$ from its unique extension. Moreover, if $F$ is an isometry with $b=0$ (see Theorem \ref{thm: Lempert_isom}), we point out that $\mathcal C_\omega$ is stable under the extension:
\begin{proposition}\label{prop: C_stable} Suppose that $F:\mathcal H_\omega \to \mathcal H_\omega$ is a differentiable $L^2$ isometry with $b=0$. Then $F(\mathcal C_\omega) \subset \mathcal C_\omega$. More importantly,  $\sup_X \|u_j -u \| \to 0$ implies  $\sup_X \|F(u_j) -F(u) \| \to 0$ for any $u_j,u \in \mathcal C_\o$.
\end{proposition}
\begin{proof}We only argue the case when $a=1$, as the proof is analogous in case $a=-1$. Since $d_2$-convergence implies pointwise a.e. convergence (see \cite[Theorem 5]{Da15}),  Proposition \ref{prop: monotonicity}(i) holds for the extension $F:\mathcal E^2_\o \to \mathcal E^2_\o$ and $u,v \in \mathcal E^2_\omega$ satisfying $u \leq v$.

 Let $u \in \mathcal C_\omega$. Then \cite{BK07} implies existence of $u_k \in \mathcal H_\omega$ such that $u_k \searrow u$. In fact, due to Dini's lemma, the convergence is uniform. From Proposition \ref{prop: monotonicity} it follows that $\{F(u_k)\}_k \subset \mathcal H_\o$ is monotone decreasing. Due to uniform convergence, we have that for any $\varepsilon >0$ there exists $k_0$ such that $u \leq u_k \leq u+ \varepsilon$ for $k \geq k_0$. Then Proposition \ref{prop: monotonicity} implies that $F(u) \leq F(u_k) \leq F(u) + \varepsilon, \ k \geq k_0$. This gives that $F(u_k)$ converges to $F(u)$ uniformly, in particular $F(u) \in \mathcal C_\omega$. 
 
Lastly, we can essentially repeat the above argument for continuous potentials $u_j$ converging uniformly to $u$, concluding the last statement of the proposition.
\end{proof}

\subsection{Examples of differentiable $L^2$ isometries on $\mathcal H_\o$} 
In this short subsection we describe  three examples of global differentiable $L^2$ isometries on $\mathcal H_\o$. Later we will argue that in fact all  isometries arise as compositions of these examples.

$\bullet$ First we take a closer look at the Monge--Amp\`ere flip $\mathcal I:\mathcal H_\omega \to \mathcal H_\o$, defined in  Section 1, perhaps first introduced in \cite{Le17a}. Let $[0,1] \ni t \to \gamma_t \in \mathcal H_\o$ be a smooth curve. Since $\frac{d}{dt}I(\gamma_t) = \int_X \dot \gamma_t \omega_{\gamma_t}^n$, we obtain that
$$\int_X \Big(\frac{d}{dt} \mathcal I(\gamma_t)\Big)^2 \omega_{\gamma_t}^n=\int_X \Big(\dot \gamma_t - 2 \int_X \dot \gamma_t \omega_{\gamma_t}^n\Big)^2= \int_X \dot \gamma_t^2  \omega_{\gamma_t}^n,$$
hence $\mathcal I$ is indeed an involutive $L^2$ isometry, with $a=1$ and $b=2$ (see Theorem \ref{thm: Lempert_isom}). This simple map has the following intriguing property, that will help to adjust the $b$ parameter of arbitrary isometries without changing the $a$ parameter:

\begin{lemma}\label{lem: b_flip} Suppose that $F: \mathcal H_\omega \to \mathcal H_\omega$ is a differentiable $L^2$ isometry. The $a$ parameter of $F$ and $F \circ \mathcal I$ is always the same. Regarding the $b$ parameter the following hold:\\
\noindent (i) If $b=0$ for $F$, then $b=2a$ for $F \circ \mathcal I$.\\ 
\noindent (ii) If $b=2a$ for $F$, then $b=0$ for $F \circ \mathcal I$.
\end{lemma}

\begin{proof}
Let $[0,1] \ni t \to \gamma_t \in \mathcal H_\omega$ be a smooth curve.  Then we have that
$$\frac{d}{dt} F (\mathcal I(\gamma_t)) = F_* (\mathcal I_* \dot \gamma_t) = F_*\Big(\dot \gamma_t - 2\int_X \dot \gamma_t \omega_{\gamma_t}^n\Big).$$
If $a=1$ and $b=0$ for $F$, then we get that $\frac{d}{dt}F (\mathcal I(\gamma_t)) = \dot \gamma_t \circ G_u - 2\int_X \dot \gamma_t \omega_{\gamma_t}^n$. If $a=-1$ and $b=0$ for $F$, then $\frac{d}{dt}F (\mathcal I(\gamma_t)) = -\dot \gamma_t \circ G_u + 2\int_X \dot \gamma_t \omega_{\gamma_t}^n$, addressing (i).

In case $a=1$ and $b=2a$ for $F$, then  $\frac{d}{dt}F (\mathcal I(\gamma_t)) = \dot \gamma_t \circ G_u$. Similarly, if 
$a=-1$ and $b=2a$ for $F$, then  $\frac{d}{dt}F (\mathcal I(\gamma_t)) = -\dot \gamma_t \circ G_u$, addressing (ii).
\end{proof}

$\bullet$ Now let $f: X \to X$ be a biholomorphism preserving the K\"ahler class $[\omega]$. Then $f$ induces a map $L_f: \mathcal H \to \mathcal H$ via pullbacks: $\omega_{L_f(u)}:= f^* \omega_u$, where we made the identification $\mathcal H \simeq I^{-1}(0)$. Using this identification it is possible to describe the action of $F$ on the level of potentials in the following manner \cite[Lemma 5.8]{DR17}:
\begin{equation}\label{eq: aut_action_DR17}
L_f(u)= L_f(0) + u \circ F, \ \ \ u \in I^{-1}(0),
\end{equation}
where $0 \in I^{-1}(0)$ is simply the zero K\"ahler potential. More importantly, $L_f$ further extends to a map $L_f:\mathcal H_\omega \to \mathcal H_\omega$ in the following manner:
$$L_f(v) = L_f(v - I(v)) + I(v), \ \ v \in \mathcal H_\omega.$$
It is well known that $L_f$ thus described gives a differentiable $L^2$ isometry of $\mathcal H_\o$ with $a=1$ and $b=0$. Actually, using the language of Theorem 2.1 applied to $L_f$, we obtain that $G_u = f$ for all $u \in \mathcal H_\omega$. We leave the related simple computation to the reader.\medskip

$\bullet$ Now let $g: X \to X$ be an anti-biholomorphism that flips the K\"ahler class $[\omega]$. By definition, such a map is a diffeomorphism  satisfying $\frac{\partial g_j}{\partial z_k} = 0$ for all $j,k \in \{1,\ldots,n\}$ in any choice of local coordinates. For example, the map $g(z)=\bar z$ is an anti-biholomorphism of the unit torus $\Bbb C / \Bbb Z[i]$ that flips that class of the flat K\"ahler metric.

Such a map $g$ induces another map $N_g: \mathcal H \to \mathcal H$ via pullbacks: $\omega_{N_g(u)}:= -g^* \omega_u$. Here we used again the identification $\mathcal H \simeq I^{-1}(0)$. Similar to \eqref{eq: aut_action_DR17}, it is possible to describe the action of $N_g$ on the level of potentials in the following manner:
\begin{equation}\label{eq: aut_anti_holo}
N_g(u)= N_g(0) + u \circ g, \ \ \ u \in I^{-1}(0).
\end{equation}
To show this, we have to go through the proof of \cite[Lemma 5.8]{DR17} in the anti-holomorphic context. As a beginning remark, we notice that $g^* \ddbar v = -\ddbar v \circ g$ for all smooth functions $v$. With this in mind, we have that
$$\omega + i\ddbar (N_g(0) + u \circ g)=- g^* \omega - g^* i\ddbar u=- g^* \omega_u=\omega_{N_g(u)}= \omega + i\ddbar N_g(u).$$ 
In particular, $N_g(0) + u \circ g - N_g(u)$ is a constant. To show that this constant is equal to zero, we only need to argue that $I(N_g(0) + u \circ g)=0= I(N_g(u))$. But this holds because of the following computation:
\begin{flalign*}
I(N_g(0) + u \circ g)&= I(N_g(0) + u \circ g)- I(N_g(0))= \frac{1}{n+1}\sum_{j=0}^n\int_X (u \circ g) \omega_{N_g(0) + u \circ g}^j \wedge \omega_{N_g(0)}^{n-j}\\
&=\frac{\pm 1}{n+1}\sum_{j=0}^n\int_X (u \circ g) g^*( \omega_{u}^j \wedge \omega^{n-j})\\
&=\frac{\pm 1}{n+1}\sum_{j=0}^n\int_X u   \omega_{u}^j \wedge \omega^{n-j}=\pm I(u)=0.
\end{flalign*}
As above, $N_g$ extends to a map $N_g:\mathcal H_\omega \to \mathcal H_\omega$ in the following manner:
$$
N_g(v) = N_g(v - I(v)) + I(v), \ \ v \in \mathcal H_\omega.
$$
We point out that $N_g$ thus described gives a differentiable $L^2$ isometry of $\mathcal H_\o$ with $a=1$ and $b=0$. To see this, let $[0,1] \ni \to \gamma_t \in \mathcal H_\o$ be a smooth curve. Using \eqref{eq: aut_anti_holo} we can write the following
$$\frac{d}{dt}N_g(\gamma_t) = \frac{d}{dt}(\gamma_t \circ g - I(\gamma_t)) + \frac{d}{dt}I(\gamma(t))= \dot \gamma_t \circ g.$$
In the language of Theorem 2.1 applied to $N_g$, we actually obtained that $G_u = g$ for all $u \in \mathcal H_\omega$.

\section{Proof of Theorem \ref{thm: no_smooth_symmetry_intr}}

The argument of Theorem \ref{thm: no_smooth_symmetry_intr} is split into two parts. First we show that there exist no global differentiable isometries with $a=-1$. Later we will classify all global differentiable isometries with $a=1$.

Before we go into specific details, we recall the following simple lemma that will be used numerous times in our arguments:

\begin{lemma}\label{lem: Da17a}\textup{\cite[Lemma 3.1]{Da17a}} Suppose that $u_0,u_1 \in \mathcal C_\omega$ and $[0,1] \ni t \to u_t \in \mathcal E^2_\omega$ is the $d_2$-geodesic connecting these potentials. Then we have that
$$\inf_X \dot u_0 = \inf_X (u_1 - u_0), \ \ \ \ \sup_X \dot u_0 = \sup_X (u_1 - u_0).$$
\end{lemma}
\begin{proof} First we argue that $ \inf_X \dot u_0 = \inf_X (u_1-u_0)$. From \eqref{eq: geod_env_def} we obtain the estimate $u_t \geq u_0 + t \inf_X (u_1-u_0), \ t \in[0,1]$. In particular, $\dot u_0 \geq \inf_X (u_1-u_0)$. Using $t$-convexity it follows that $u_t(y) = u_0(y) + t \inf_X (u_1-u_0)$ for $y \in X$ such that $u_1(y)-u_0(y)=\inf_X (u_1-u_0)$. This implies that $t \to u_t(y)$ is linear, implying that $ \inf_X \dot u_0 = \inf_X (u_1 - u_0)$.

For the second identity, we notice that $t$-convexity implies $\sup_X \dot u_0 \leq \sup_X (u_1 - u_0)$. In addition, \eqref{eq: geod_env_def} implies that $u_1 - (1-t)\sup_X (u_1-u_0) \leq u_t, \ t \in[0,1]$. Relying on $t$-convexity again, we obtain that $\dot u_0(z)= u_1(z) - u_0(z) =\sup_X (u_1 - u_0)$, for $z \in X$ with $u_1(z)-u_0(z) =\sup_X (u_1-u_0)$. Summarizing, we obtain that $\sup_X \dot u_0 = \sup_X (u_1 - u_0)$, as desired.
\end{proof}
\subsection{Isometries with $a=-1$} We start with a lemma: 
\begin{lemma}\label{lem: inf_sup} Suppose that $F:\mathcal H_\omega \to \mathcal H_\omega$ is a differentiable $L^2$ isometry with $a=-1$ and $b=0$. Let $\phi \in \mathcal H_\o$ and $u \in \mathcal H_\o$ with $u \leq \phi$. Then we have that $F(u) \geq F(\phi)$ and
\begin{equation}\label{eq: inf_sup_id}
\sup_X (F(u)-F(\phi)) = -\inf_X (u-\phi).
\end{equation}
\end{lemma}
\begin{proof} That $F(u) \geq F(\phi)$ follows from Proposition \ref{prop: monotonicity}(ii). As it is pointed out on \cite[p.2]{Le17a}, Theorem \ref{thm: Lempert_isom}  implies that $F$ is a $d_p$-isometry for any $p \geq 1$. This implies that $d_p(\phi,u) = d_p(F(\phi),F(u))$ for any $p \geq 1$. 

Let $[0,1] \ni t \to u_t,v_t \in \mathcal H^{1,1}_\omega$ be the $C^{1,1}$ geodesic connecting $u_0:=\phi,u_1:=u$, respectively $v_0:=F(\phi)$ and $v_1:= F(u)$. By the comparison principle for weak geodesics (see for example \cite[Proposition 2.2]{Bl12}) it follows that $v_t \geq F(\phi)$ and $u_t \leq \phi$ for any $t \in [0,1]$. In particular, $\dot v_0 \geq 0$ and $\dot u_0 \leq 0$. 

Using \cite[Theorem 1]{Da15} we arrive at:
$$\int_X |\dot u_0|^p \omega_\phi^n = d_p(\phi,u)^p = d_p(F(\phi),F(u))^p = \int_X |\dot v_0|^p \omega_{F(\phi)}^n, \ p \geq 1.$$
Raising  to the $\frac{1}{p}$-power, and letting $p \to \infty$ gives  that 
\begin{equation}\label{eq: dot_id}
\sup_X \dot v_0 = - \inf_X \dot u_0. 
\end{equation}
From Lemma \ref{lem: Da17a} we get that $ \inf_X \dot u_0 = \inf_X (u-\phi)$ and $\sup_X \dot v_0 =\sup_X (F(u) - F(\phi))$. Putting this together with \eqref{eq: dot_id}, we obtain \eqref{eq: inf_sup_id}, as desired.
\end{proof}

\begin{theorem}\label{thm: no_smooth_symmetry_a=-1} There exists no differentiable $L^2$ isometry $F:\mathcal H_\o \to \mathcal H_\o$ with $a=-1$.
\end{theorem}

We note that this result already implies Corollary \ref{cor: no_smooth_symmetry_intr1}.
 
\begin{proof} Due to Lemma \ref{lem: b_flip}, after possibly composing $F$ with $\mathcal I$, we only need to worry about the case $a=-1$ and $b=0$. 

Since $F:\mathcal H_\o \to \mathcal H_\o$ is a differentiable $L^2$-isometry, it is also a $d_2$-isometry, hence it extends to a unique $d_2$-isometry $F:\mathcal E^2_\o \to \mathcal E^2_\o$.

Let $\phi \in \mathcal H_\o$. Let $u \in \mathcal E^2_\o \setminus L^\infty$ with $u \leq \phi -1$, and we choose $u_k \in \mathcal H_\o$ such that $u_k \searrow u$ and $u_k \leq \phi$. Such a sequence can always be found \cite{BK07}.

Due to our choice of $u$ we have that $\inf_X (u_k-\phi) \searrow -\infty$. From Lemma \ref{lem: inf_sup} it follows that $\sup_X F(u_k)=\sup_X (F(u_k)-F(\phi)) \nearrow +\infty$. Since $F$ is a $d_2$-isometry, we have that $d_2(F(u_k),F(u))= d_2(u,u_k) \to 0$. However \cite[Theorem 5(i)]{Da15} gives that $\sup_X F(u_k) \to \sup_X F(u) <+\infty$, which is a contradiction.
\end{proof}

\subsection{Isometries with $a=1$}  

To start, we point out an important relationship between $d_2$-geodesics and  differentiable $L^2$ isometries with $a=1$ and $b=0$:

\begin{proposition}\label{prop: geod_tangent_a=1}Suppose that $F:\mathcal H_\omega \to \mathcal H_\omega$ is a differentiable $L^2$ isometry with $a=1$ and $b=0.$ Let $[0,1] \ni t \to u_t \in \mathcal E^2_\omega$ be the $d_2$-geodesic connecting $u_0 \in \mathcal H_\omega$ and $u_1 \in \mathcal C_\omega$. Then 
\begin{equation}\label{eq: geod_tanget_eq_a=1}
\dot u_0\circ G_{u_0}= \dot{F(u_0)}.
\end{equation}
\end{proposition}
Here and below $\dot u_0 := \frac{d}{dt}\big|_{t=0}F(u_t) $ and $\dot {F(u_0)}:=\frac{d}{dt}\big|_{t=0}F(u_t)$ are the initial tangent vectors of the $d_2$-geodesics $t \to u_t$ and  $t \to F(u_t)$, interpreted according to the discussion preceding Theorem \ref{thm: no_metric_symmetry_intr}. 
\begin{proof} There exists a constant $c \in \Bbb R$ such that $u_0 > u_1 + c$. Since $F(u_t + tc) = F(u_t) + tc$ (Proposition \ref{prop: monotonicity}(i)), we can assume without loss of generality that $u_0 >  u_1$.

First, we show \eqref{eq: geod_tanget_eq_a=1} in case $u_1 \in \mathcal H_\o$. Let $[0,1] \ni u^\varepsilon_t \in \mathcal H_\o$ be the smooth $\varepsilon$-geodesics of X.X. Chen, connecting $u_0$ and $u_1$ \cite{Ch00}. It is well known that $u^\varepsilon_t \nearrow u_t$ as $\varepsilon \to 0$, where $t \to u_t$ is the $C^{1,1}$-geodesic joining $u_0$ and $u_1$. Due to Proposition \ref{prop: monotonicity} and Proposition \ref{prop: C_stable}, for the curves $t \to F(u^\varepsilon_t),F(u_t)$ we obtain that $F(u^\varepsilon_t) \nearrow F(u_t)$. Since $t \to F(u^\varepsilon_t)$ is a $C^1$ curve, via Theorem \ref{thm: Lempert_isom}, we obtain that 
$$\dot u^\varepsilon_0\circ G_{u_0} = \dot{F(u_0^\varepsilon)} \leq \dot{F(u_0)} \leq 0, \ \ \varepsilon >0.$$
Taking the limit $\varepsilon \to 0$, since $u^\varepsilon \to_{C^{1,\alpha}} u$, we arrive at $\dot u_0\circ G_{u_0} \leq \dot{F(u_0)} \leq 0$. By Theorem \ref{thm: Lempert_isom} we have that $G_{u_0}^* \omega_{u_0}^n= \pm \omega_{F(u_0)}^n$. Using this and \cite{Ch00} (see also \cite[Theorem 1]{Da15})  we obtain that 
$$\int_X (\dot u_0\circ G_{u_0})^2 \omega_{F(u_0)}^n = \int_X \dot u_0^2 \omega_{u_0}^n=d_2(u_0,u_1)^2=d_2(F(u_0),F(u_1))^2=\int_X \dot{F(u_0)}{}^2 \omega_{F(u_0)}^n.$$
Due to continuity we conclude that $\dot u_0\circ G_{u_0} = \dot{F(u_0)}$, as desired.

Now we treat the general case. Let $u^k_1 \in \mathcal H_\o, \ k \in \Bbb N$ such that $u_0 > u^k_1$ and $u^k_1 \searrow u_1 \in \mathcal C_\o$. Also, by $[0,1] \ni t \to u_t,u^k_t \in \mathcal E^2_\omega$ we denote the $d_2$-geodesics connecting $u_0$ and $u_1$, respectively $u_0$ and $u^k_1$. Since $F$ is a $d_2$-isometry, we obtain that $[0,1] \ni t \to F(u_t),F(u^k_t) \in \mathcal E^2_\omega$ are the $d_2$-geodesics connecting $F(u_0)$ and $F(u_1)$, respectively $F(u_0)$ and $F(u^k_1)$. Due to $t$-convexity, $k$-monotonicity and Proposition \ref{prop: monotonicity}, we obtain that $\dot u_0^k \searrow \dot u_0$ and $\dot{F(u_0^k)} \searrow \dot{F(u_0)}$. Letting $k \to \infty$ we arrive at the desired conclusion: $\dot u_0\circ G_{u_0} = \lim_k (\dot u_0^k\circ G_{u_0})=\lim_k \dot{F(u_0^k)} = \dot{F(u_0)}$.
\end{proof}

This result together with Lemma \ref{lem: Da17a} gives the following corollary, paralleling Lemma \ref{lem: inf_sup}:
\begin{corollary} Suppose that $F:\mathcal H_\omega \to \mathcal H_\omega$ is a differentiable $L^2$ isometry with $a=1$ and $b=0$. Suppose that $u,v \in \mathcal C_\omega$. Then we have that $F(u) ,F(v) \in \mathcal C_\omega$ and 
\begin{equation}\label{eq: inf_isometry_id}
\inf_X (F(u) - F(v)) = \inf_X (u-v).
\end{equation}
\end{corollary}
By the switching the role of $u$ and $v$, we obtain that the above identity holds for the suprema as well.
\begin{proof} That $F(u) ,F(v) \in \mathcal C_\omega$, follows from Proposition \ref{prop: C_stable}. First we deal with the case when $u,v \in \mathcal H_\o$. If $[0,1] \ni t \to h_t \in \mathcal H_\o$ is the $C^{1,1}$-geodesic connecting $h_0:=u$ and $h_1:=v$, then Lemma \ref{lem: Da17a} gives that 
$$\inf_X (v-u) = \inf_X \dot h_0 \ \ \textup{ and } \ \ \inf_X (F(v)-F(u)) = \inf_X \dot{F(h_0)}.$$
Putting this together with \eqref{eq: geod_tanget_eq_a=1}, we obtain that $\inf_X (v-u) = \inf_X (F(v)-F(u))$, as desired.

When $u,v \in \mathcal C_\omega$, by \cite{BK07} one can find $u^k,v^k \in \mathcal H_\omega$ such that $\sup_X |u^k-u| \to 0$ and $\sup_X |v^k-v| \to 0$. Then Proposition \ref{prop: C_stable} implies that $\sup_X |F(u^k)-F(u)| \to 0$ and $\sup_X |F(v^k)-F(v)| \to 0$. 

By uniform convergence we have $\inf_X (u^k - v^k) \to \inf_X (u - v)$ and $\inf_X (F(u^k) - F(v^k)) \to \inf_X (F(u) - F(v))$. The conclusion follows after taking the $k$-limit of $\inf_X (u^k - v^k) = \inf_X (F(u^k) - F(v^k))$.  
\end{proof}

To continue, we need an an auxiliary construction. Fixing $x \in X$ and a small enough coordinate neighborhood $O_x \subset X$, we can find a function $\rho_x \in C^\infty(X)$ such that $\rho_x(y) = e^{\frac{-1}{\|y-x\|^2}}$ for all $y \in O_x$, and there exists $\beta>0$ such that $\beta \leq \rho_x(y) \leq 1$ for all $y \in X \setminus O_x$. 

\begin{proposition}\label{prop: delta_x_subgeod} For $u \in \mathcal H_\o$ and $x \in X$ there exists $\delta >0$ such that $[0,1] \ni t \to u_t := u + \delta (t + \frac{t^2}{2}) \rho_x \in \mathcal H_\o$ is a subgeodesic.
\end{proposition}

\begin{proof} Let $U(s,y) = u_{\textup{Re }s}(y) \in C^\infty(S \times X)$, where $S = \{ 0 \leq \textup{Re }z \leq 1
\} \subset \Bbb C$. It is clear that for small enough $\delta >0$ we have that $u_t \in \mathcal H_\omega, \ t \in [0,1]$. More precisely, there exists $\alpha>0$ such that $\omega_{u_t} \geq \alpha \omega, \ t \in [0,1]$.

This implies that $\omega + i \partial_{S \times X} \bar \partial_{S \times X} U$ has at least $n$ non-negative eigenvalues for all $(s,y) \in S \times X$. To conclude that $\omega + i \partial_{S \times X} \bar \partial_{S \times X} U \geq 0$ it is enough to show that the determinant of this Hermitian form is non-negative. This is equivalent with
$\ddot{u}_t - \langle \partial \dot u_t, \bar \partial \dot u_t\rangle_{\omega_{u_t}} \geq 0$ on $[0,1] \times X.$
To show this, we start the following sequence of estimates:
\begin{flalign*}
\ddot{u}_t - \langle \partial \dot u_t, \bar \partial \dot u_t\rangle_{\omega_{u_t}} & = \delta \rho_x -  \delta^2(1+t)^2 \langle \partial \rho_x, \bar \partial \rho_x \rangle_{\omega_{u_t}} \geq \delta \rho_x -  \frac{\delta^2(1+t)^2}{\alpha} \langle \partial \rho_x, \bar \partial \rho_x \rangle_{\omega}.
\end{flalign*}
After possibly shrinking $\delta \in (0,1)$, we obtain that it is enough to conclude that the last expression is non-negative on the neighborhood $O_x$, where know that $\rho_x(y) = e^{\frac{-1}{\|y-x\|^2}}, \ y \in O_x$. In particular, on $O_x \setminus \{x\}$ we have that $\langle \partial \rho_x, \bar \partial \rho_x \rangle_{\omega}/ \rho_x \simeq e^{\frac{-1}{\|y-x\|^2}}\frac{1}{\|y-x\|^6}$, which is uniformly bounded. In particular, after possibly further shrinking $\delta \in (0,1)$  we obtain that 
\begin{flalign*}
\ddot{u}_t - \langle \partial \dot u_t, \bar \partial \dot u_t\rangle_{\omega_{u_t}} &  \geq \delta \rho_x -  \frac{\delta^2(1+t)^2}{\alpha} \langle \partial \rho_x, \bar \partial \rho_x \rangle_{\omega} \geq 0,
\end{flalign*}
what we desired to prove.
\end{proof}

\begin{theorem}\label{thm: no_smooth_symmetry_a=1} Suppose that $F:\mathcal H_\o \to \mathcal H_\o$ is a differentiable $L^2$ isometry with $a=1$. Then exactly one of the following holds:\\
\noindent (i) $F$ is induced by a biholomorphism or anti-biholomorphism $f:X \to X$ that preserves or flips the K\"ahler class $[\omega]$, respectively.\\
\noindent (ii)  $F \circ \mathcal I$ is induced by a biholomorphism or anti-biholomorphism $f:X \to X$ that preserves or flips the K\"ahler class $[\omega]$, respectively.
\end{theorem}

\begin{proof} Due to Lemma \ref{lem: b_flip}, after possibly composing $F$ with $\mathcal I$, we only need to worry about the case $a=1$ and $b=0$. In this case we will show that $F$ is induced by a biholomorphism or anti-biholomorphism $g:X \to X$ that preserves or flips the K\"ahler class $[\omega]$.

In the language of Theorem \ref{thm: Lempert_isom} applied to $F$, the first step is to show that $G_u = G_v$ for all $u,v \in \mathcal H_\omega$.
 
We fix $x \in X$ and $u,v \in \mathcal H_\omega$. We will show that $G_u^{-1}(x)=G_v^{-1}(x)$. Since $G_{u+c} = G_{u}$   for any $c \in \Bbb R$ (Corollary \ref{cor: G_constant}), we can assume that $u(x)=v(x)$.  First we prove that $G_u^{-1}(x)=G_v^{-1}(x)$ under the extra non-degeneracy condition $\nabla u(x) \neq \nabla v(x)$.

Let $\eta >0$ be such that $w:=\max(u,v) + \eta \rho_x \in \mathcal C_\omega$. From our setup it is clear that $w \geq \max(u,v)$, and the graphs of  $w$, $u$ and $v$ only meet at $x$. Extending the isometry $F$ to the metric completion, Proposition \ref{prop: monotonicity} and Proposition \ref{prop: C_stable} implies that $F(w) \geq \max(F(u),F(v)), \ F(w)\in \mathcal C_\omega$ and $F(u),F(v) \in \mathcal H_\o$. Below we will show that $F(w)$ and $F(u)$ only meet at $G^{-1}_u(x)$, moreover $F(w)$ and $F(v)$ only meet at $G^{-1}_v(x)$. Finally, we will show that the graphs of $F(w)$, $F(u)$ and $F(v)$ have to meet at some point of $X$, implying that  $G^{-1}_u(x)=G^{-1}_v(x)$,  as desired.

Let us denote by $[0,1] \ni t \to u_t,v_t \in \mathcal E^2_\o$ the $d_2$-geodesics joining $u_0:=u$ with $u_1 := w$, respectively $v_0:=v$ with $v_1 := w$. From Proposition \ref{prop: geod_tangent_a=1} it follows that 
\begin{equation}\label{eq: F_dot_eq}
\dot{F(u_0)} = \dot u_0\circ G_u, \ \ \ \dot{F(v_0)} = \dot v_0\circ G_v.
\end{equation}

Using \eqref{eq: geod_env_def} there exists  a small enough $\delta >0$ in the statement of Proposition \ref{prop: delta_x_subgeod} such that $u + \delta (t + \frac{t^2}{2})\rho_x \leq u_t$ and $v + \delta (t + \frac{t^2}{2})\rho_x \leq v_t, \ t \in [0,1]$. Using this, $t$-convexity and \eqref{eq: F_dot_eq}, we obtain that 
$$
F(w) - F(u) \geq \dot{F(u_0)} = \dot u_0\circ G_u  \geq  \delta \rho_x\circ G_u, \ \ \ F(w) - F(v) \geq\dot{F(v_0)} = \dot v_0\circ G_v \geq \delta \rho_x\circ G_v.
$$
Due to \eqref{eq: inf_isometry_id} these two estimates imply the existence of a unique $y \in X$ and a unique $z \in X$ such that
\begin{equation}
F(w)(y) - F(u)(y)=0  \ \ \textup{ and } \ \ F(w)(z) - F(v)(z)=0.
\end{equation}
In fact, we need to have that $y = G^{-1}_u(x)$ and $z = G^{-1}_v(x)$. In particular, the graphs of $F(w)$ and $F(u)$ only meet at $y$, and graphs of $F(w)$ and $F(u)$ only meet at $z$.

In case $y \neq z$, uniqueness of $y$ and $z$ implies that $y \in \{F(u) > F(v)\}$ and  $y \in \{F(v) > F(u)\}$ (recall that $F(w) \geq \max(F(u),F(v)$). This implies that the graphs of $F(w)$ and $\max(F(u),F(v))$ meet at only two points ($y$ and $z$), away from the compact set $\{F(u)=F(v)\}$. Consequently, using classical Richberg approximation \cite[Chapter I, Lemma 5.18]{Dem12}, one can take a ``regularized maximum" of $F(u)$ and $F(v)$ to obtain $\beta \in \mathcal H_\omega$ satisfying
$$F(w)  \geq \beta \geq \max(F(u),F(v)).$$
Since $F:\mathcal H_\omega \to \mathcal H_\omega$ is surjective, there exists a unique $\alpha \in \mathcal H_\o$ s.t. $F(\alpha)=\beta$. Using \eqref{eq: inf_isometry_id} again, we obtain that 
$$\max(u,v) + \delta \rho_x  =w\geq \alpha \geq \max(u,v).$$
Since $\nabla u(x) \neq \nabla v(x)$ and $w(x)=\alpha(x)=\max(u,v)(x)$, this is a contradiction with the smoothness of $\alpha$ at $x$. Consequently, we need to have that $G^{-1}_u(x)=y = z =G^{-1}_v(x)$, as desired.

In case $\nabla u(x) = \nabla v(x)$, one finds $q \in \mathcal H_\omega$ (via small perturbation) such that $u(x)=v(x)=q(x)$ and $\nabla u(x) \neq  \nabla q(x)$ along with $\nabla v(x) \neq  \nabla q(x)$. Then by the above we have that $G^{-1}_u(x)=G^{-1}_q(x)$ and $G^{-1}_v(x)=G^{-1}_q(x)$, ultimately giving that $G^{-1}_u(x)=G^{-1}_v(x)$ for any $u,v \in \mathcal H_\o$.

Using Theorem \ref{thm: Lempert_isom}, an integration along the curve $t \to tu$ gives that
\begin{equation}\label{eq: F_formula}
F(u)-F(0) = \int_0^1 (u \circ g) dt = u \circ g, \ \ \ u \in \mathcal H_\omega.
\end{equation}
Returning to the statement of Theorem \ref{thm: Lempert_isom}, we either have $g^* \omega_u = \omega_{F(u)}, \ u \in \mathcal H_\omega$, or $g^* \omega_u = -\omega_{F(u)}, \ u \in \mathcal H_\omega$.

Assuming that $g^* \omega_u = \omega_{F(u)}$, using \eqref{eq: F_formula} we arrive at the identity $g^* (i\ddbar u) = i \ddbar (u \circ g)$. Since after a dilation all elements of $C^\infty(X)$ land in $\mathcal H_\o$, we obtain that actually $g^* (i\ddbar v) = i \ddbar (v \circ g)$ for all $v \in C^\infty(X)$. According to the next lemma $g$ has to be holomorphic, implying that $F=L_g$ (see Section 2.3).

In case $g^* \omega_u = -\omega_{F(u)}$, by a similar calculation we arrive at $g^* (i\ddbar v) = -i \ddbar (v \circ g)$ for all $v \in C^\infty(X)$.  According to the next lemma $g$ has to be anti-holomorphic, giving that $F=N_g$ (see Section 2.3), finishing the proof.
\end{proof}

\begin{lemma} Suppose that $g: X \to X$ is a smooth map.\\
\noindent (i) If $i\ddbar (u \circ g) = g^* (i\ddbar u)$ for all $u \in C^\infty(X)$ then $g$ is holomorphic.\\
\noindent (ii) If $i\ddbar (u \circ g) = -g^* (i\ddbar u)$ for all $u \in C^\infty(X)$ then $g$ is anti-holomorphic.
\end{lemma}
\begin{proof} We only show (i) as the proof of (ii) is analogous. We start with the following  computations expressed in local coordinates:
\begin{flalign}\label{eq: comp_ddbar}
i \ddbar (u\circ g)  =& i \frac{\partial^2 (u \circ g)}{\partial z_j \partial \overline{z_k}} dz_j \wedge d \overline{z_k} \nonumber \\
=& i\frac{\partial^2 u}{\partial z_a \partial \overline{z_b}} \bigg[\frac{\partial g_a}{\partial z_j}\frac{\partial \overline{g_b}}{\partial \overline{z_k}} + \frac{\partial g_a}{\partial \overline{ z_k}} \frac{\partial \overline{ g_b}}{\partial z_j}\bigg]dz_j \wedge d \overline{z_k}  \nonumber\\
&+ i\frac{\partial^2 u}{\partial z_a \partial {z_b}} \frac{\partial g_a}{\partial z_j}\frac{\partial {g_b}}{\partial \overline{z_k}} dz_j \wedge d \overline{z_k} +  i\frac{\partial^2 u}{\partial \overline{z_a} \partial \overline{z_b}} \frac{\partial \overline{g_a}}{\partial z_j}\frac{\partial \overline{g_b}}{\partial \overline{z_k}} dz_j \wedge d \overline{z_k}\\
&+ i \frac{\partial u}{\partial z_b} \frac{\partial^2 g_b}{\partial z_j \overline{\partial z_k}} dz_j \wedge d \overline{z_k} + i \frac{\partial u}{\partial \overline{z_b}} \frac{\partial^2 \overline{g_b}}{\partial z_j \overline{\partial z_k}} dz_j \wedge d \overline{z_k}.\nonumber
\end{flalign}
Knowing that $g^* (i\ddbar u)$ is a $(1,1)$ form we also have that
\begin{flalign}\label{eq: pullback_ddbar}
g^* (i\ddbar u) & = i\frac{\partial^2 u}{\partial z_a \partial \overline{ z_b}} \bigg[\frac{\partial g_a}{\partial z_j}\frac{\partial \overline{g_b}}{\partial \overline{z_k}} -  \frac{\partial g_a}{\partial \overline{z_k}}\frac{\partial \overline{g_b}}{\partial z_j}\bigg]dz_j \wedge d \overline{z_k}.
\end{flalign}
Clearly, it is enough to show that $g$ is holomorphic in local coordinate charts. By linearity we can assume that $i\ddbar (u \circ g) = g^* (i\ddbar u)$ holds for complex valued smooth functions $u$. 

Let $x \in X$, and we pick $u$ such that in a coordinate neighborhood of $x$  we have that $u(z) = z_b, \ b \in \{1,\ldots, n\}$. Then $i\ddbar (u \circ g) = g^* (i\ddbar u)$ gives that ${\partial^2 g_b}/{\partial z_j \partial \overline{z_k}}=0$ for all $j,k \in \{1,\ldots, n\}$ at  $x$. Similarly, after choosing $u(z) = \overline{z_b}, \ b \in \{1,\ldots, n\}$ in a coordinate neighborhood of $x$, we obtain that ${\partial^2 \overline{g_b}}/{\partial z_j \partial \overline{z_k}}=0$ for all $j,k \in \{1,\ldots, n\}$ at  $x$. Since $x \in X$ was arbitrary, the terms in the last line  of \eqref{eq: comp_ddbar} vanish for any choice of $u$. 

Repeating this process for $u(z) = z_a  z_b$ and $u(z)=\bar z_a  \bar z_b$, we conclude that the terms in the second line of \eqref{eq: comp_ddbar} vanish as well, for any choice of $u$.

Revisiting the identity $i\ddbar (u \circ g) = g^* (i\ddbar u)$ one more time, after picking $u$ such that $i\ddbar u$ is positive definite in a neighborhood of $x\in X$, we obtain that $\partial  g_a/\partial \overline{z_j} =0$ for any $a,j \in  \{1,\ldots, n\}$ at $X$, implying that $g$ is indeed holomorphic.
\end{proof}
\section{Proof of Theorem \ref{thm: no_metric_symmetry_intr} and \ref{thm: no_geod_ext_intr}}

We start with a lemma about the concatenation of geodesics in $\mathcal E^2_\o$:

\begin{lemma}\label{lem: concat_geod} Suppose that $[-1,0] \ni  t\to v_t \in \mathcal E^2_\o$  and $[0,1] \ni  t\to u_t \in \mathcal E^2_\o$ are $d_2$-geodesics such that $u_0=v_0 \in \mathcal H_\o$ and $\dot u_0=\dot v_0 \in L^2(\omega^n)$. Then $[-1,1] \ni t \to w_t \in \mathcal E^2_\o$, the concatenation of the curves $t \to u_t$ and $t \to v_t$, is the $d_2$-geodesic joining $v_{-1},u_1 \in \mathcal E^2_\o$.
\end{lemma}
\begin{proof}  By possibly changing the background metric, we can assume that $u_0=v_0=0$. From the $L^2$ version of \cite[Lemma 3.4(ii)]{BDL16}, (whose proof is identical to the $L^1$ version, presented in \cite{BDL16}) we have that
\begin{equation}\label{eq: init_length}
d_2(v_{-1},0)^2 = \int_X |\dot u_0|^2 \omega^n = \int_X |\dot v_0|^2 \omega^n = d_2(0,u_1)^2.
\end{equation}
Next we point out that 
\begin{equation}\label{eq: weak_geod_id}
d_2(v_{-1},u_1) = d_2(v_{-1},0) + d_2(0,u_1). 
\end{equation}
Indeed, from the triangle inequality we have that $d_2(v_{-1},u_1) \leq  d_2(v_{-1},0) + d_2(0,u_1)$. The reverse inequality follows from \eqref{eq: init_length} and \cite[Theorem 3.1]{DL18}:
$$d_2(v_{-1},0) + d_2(0,u_1)=\bigg(\int_X |2 \dot u_0|^2 \omega^n\bigg)^{\frac{1}{2}}  \leq d_2(v_{-1},u_1).$$
Due to uniqueness of $d_2$-geodesic segments, we only need to show that for any $a,b \in [-1,1]$ with $a<b$ we have that 
\begin{equation}\label{eq: w_geod}
d_2(w_a,w_b) = \frac{b-a}{2} d_2(v_{-1},u_1)=(b-a)d_2(0,u_1)=(b-a)d_2(v_{-1},0).
\end{equation}
Since $t \to u_t$ and $t \to v_t$ are $d_2$-geodesics, we only need to treat the case $a \in [-1,0]$ and $b \in [0,1]$. The proof of this is almost identical to that of \eqref{eq: weak_geod_id}. Indeed after another application of \cite[Theorem 3.1]{DL18} we arrive at
$$d_2(v_a,u_b) \geq \bigg( \int_X |(b-a) \dot u_0|^2 \omega^n\bigg)^{\frac{1}{2}} = (b-a) d_2(0,u_1).$$
The reverse inequality follows from the triangle inequality:
$d_2(v_a,u_b) \leq d_2(v_a,0) + d_2(0,u_b) = (b-a)d_2(0,u_1)$.
\end{proof}

\begin{proof}[Proof of Theorem \ref{thm: no_geod_ext_intr}]
By changing the background metric, we can assume without loss of generality that $\phi_0 =0$. From \eqref{eq: geod_env_def} it follows that $t \to \phi_t + Ct$ is a   $d_2$-geodesic for any $C \in \Bbb R$. As a result, we can also assume that $\phi_1 \leq 0$.

To derive a contradiction, let us further assume that there exists a $d_2$-geodesic $[-\varepsilon,1] \ni t \to \phi_t \in \mathcal E^2_\o$, as described in the statement of the theorem.

First we show that $\phi_{-\varepsilon} \geq 0$.
 This is a simple consequence of the $t$-convexity. By the results of \cite{Da17b} (see the discussion near \eqref{eq: limit_t_0_1}) there exists a set $Z \subset X$ of measure zero such that for all $x \in X \setminus Z$ we have that $t \to \phi_t(x)$ is convex,  
$\phi_0(x)=0$, $\lim_{t \nearrow 1} \phi_t(x) = \phi_1(x) \leq 0$, and $\lim_{t \searrow -\varepsilon} \phi_t(x) = \phi_{-\varepsilon}(x)$. Due to $t$-convexity, we obtain that $\phi_{-\varepsilon}(x) \geq 0$ away from $Z$. As $\phi_{-\varepsilon} \in \textup{PSH}(X,\omega)$, we obtain that $\phi_{-\varepsilon} \geq 0$. 

Since $\phi_{-\varepsilon}$ is usc, it follows that $\sup_X \phi_{-\varepsilon} <+\infty$, i.e., $\phi_{-\varepsilon} \in L^\infty$. Using \eqref{eq: geod_env_def} for the $d_2$-geodesic joining $\phi_{-\varepsilon}$ and $\phi_0$, it follows that
$$\phi_t \geq \phi_{-\varepsilon} - \frac{\varepsilon-t}{\varepsilon} \sup_X \phi_{-\varepsilon}, \ t \in [-\varepsilon,0).$$
Since $(-\varepsilon,1) \ni t \to \phi_t(x)$ is $t$-convex for all $x \in X \setminus Z$, it follows that the above estimate extends to $t \in [-\varepsilon,1]$, contradicting the fact that $\phi_1 \in \mathcal E^2_\omega \setminus L^\infty$.
\end{proof}

\begin{proof}[Proof of Theorem \ref{thm: no_metric_symmetry_intr}] We can assume without loss of generality that $\phi=0$.

To derive a contradiction, we further assume that  there exists a metric $L^2$ symmetry $F: \mathcal V \to \mathcal V$, as described in the statement of the theorem. 

Since $\mathcal V$ is $d_2$-open, it follows that $0 \in B(0,\delta) \subset \mathcal V$ for some $\delta >0$, where $B(0,\delta)$ is the $d_2$-ball of radius $\delta$ centered at $0$. As $F$ is a metric $L^2$ symmetry it follows that $F:B(0,\delta) \to B(0,\delta)$ is bijective.

Let $\psi_1 \in B(0,\delta)$ such that $\psi_1 \in \mathcal E^2_\o \setminus L^\infty$. One can find such $\psi$ as a consequence of \cite[Theorem 3]{Da15}. Let $[0,1] \ni t \to \psi_t,F(\psi_t) \in B(0,\delta)$ be the $d_2$-geodesics connecting $0$ and $\psi_1$,  respectively $0$ and $F(\psi_1)$. 

Since $F$ is a metric $L^2$ symmetry, by definition we have that $\dot \psi_0 = -\dot {F(\psi_0)}$. Consequently,  according to Lemma \ref{lem: concat_geod}, the concatenation $[-1,1] \ni t \to w_t \in B(0,\delta)$ of the curves $t \to F(\psi_{-t})$ and $t \to \psi_t$ is a $d_2$-geodesic. But then $t \to w_t$ extends $t \to \psi_t$ at $t=0$, giving a contradiction with Theorem \ref{thm: no_geod_ext_intr}.
\end{proof}

\footnotesize
\let\OLDthebibliography\thebibliography 
\renewcommand\thebibliography[1]{
  \OLDthebibliography{#1}
  \setlength{\parskip}{1pt}
  \setlength{\itemsep}{1pt}
}

\bigskip
\normalsize
\noindent{\sc University of Maryland}\\
\vspace{0.2cm}\noindent {\tt tdarvas@math.umd.edu}
\end{document}